\documentclass[12pt,a4paper]{article}

\usepackage[utf8]{inputenc}
\usepackage[T1]{fontenc}
\usepackage{lmodern}

\usepackage{amscd}
\usepackage{amsfonts}
\usepackage{amsmath}
\usepackage{amssymb}
\usepackage{amstext}
\usepackage{amsthm}
\usepackage[small,margin=5mm]{caption}
\usepackage{chngcntr}
\usepackage{color}
\usepackage{dsfont}
\usepackage{graphicx}
\usepackage{mathrsfs}
\usepackage{pifont}
\usepackage{psfrag}
\usepackage{setspace}
\usepackage[active]{srcltx}
\usepackage{stmaryrd}
\usepackage[normalem]{ulem}

\numberwithin{equation}{section}
\numberwithin{figure}{section}

\usepackage[hidelinks,colorlinks=true,urlcolor=blue,linkcolor=black,citecolor=black]{hyperref}
\newcommand*{\doi}[1]{\href{http://dx.doi.org/\detokenize{#1}}{doi}}
\allowdisplaybreaks

\newcommand{\Z}{\mathbb{Z}}
\newcommand{\R}{\mathbb{R}}
\newcommand{\W}{\mathcal{W}}
\renewcommand{\t}{{\textnormal{\emph{\textbf t}}}}

\newcommand{\x}{{\textnormal{\emph{\textbf x}}}}
\newcommand{\y}{{\textnormal{\emph{\textbf y}}}}
\newcommand{\z}{{\textnormal{\emph{\textbf z}}}}
\renewcommand{\d}{{\textnormal{\emph{\textbf d}}}}

\newcommand{\one}{\mathbf{1}}
\newcommand{\I}{\mathbf{1}}

\newcommand{\cP}{\mathcal{P}}
\newcommand{\B}{\mathcal{B}}

\theoremstyle{plain}
\newtheorem{theo}[equation]{Theorem}
\newtheorem{lemma}[equation]{Lemma}
\newtheorem{prop}[equation]{Proposition}

\theoremstyle{definition}
\theoremstyle{remark}

\newtheorem*{claim*}{Claim} 
\newtheorem*{remark*}{Remark}

\renewcommand{\geq}{\geqslant}
\renewcommand{\ge}{\geqslant}
\renewcommand{\leq}{\leqslant}
\renewcommand{\le}{\leqslant}
\renewcommand{\subset}{\subseteq}

\renewcommand{\eqref}[1]{(\ref{#1})}

\begin{document}

\title{Slow-to-Start Traffic Model: \\ Condensation, Saturation and Scaling Limits}
\author{Pablo A. Ferrari and Leonardo T. Rolla \\ {\small University of Buenos Aires}}

\maketitle

\begin{abstract}
We consider a one-dimensional traffic model with a slow-to-start rule.
The initial position of the cars in $\R$ is a Poisson process of parameter $\lambda$.
Cars have speed 0 or 1 and travel in the same direction.
At time zero the speed of all cars is 0; each car waits an exponential time to switch speed from $0$ to $1$ and stops when it collides with a stopped car.
When the car is no longer blocked, it waits a new exponential time to assume speed one, and so on.
We study the emergence of condensation for the saturated regime $\lambda>1$ and the critical regime $\lambda=1$, showing that in both regimes all cars collide infinitely often and each car has asymptotic mean velocity $1/\lambda$.
In the saturated regime the moving cars form a point process whose intensity tends to 1.
The remaining cars condensate in a set of points whose intensity tends to zero as $1/\sqrt t$.
We study the scaling limit of the traffic jam evolution in terms of a collection of coalescing Brownian motions.
\end{abstract}

% \clearpage
\section{Introduction}

We consider a system of cars that move from right to left in $\R$.
The initial
car positions form a homogeneous Poisson process of parameter $\lambda$ on $\R$.
Cars are labeled in increasing order.
Each car has two possible
velocities, either $0$ or $1$, and they cannot overpass.
Initially all cars have
speed 0.
For each $i\in\Z$, car $i$ waits an exponentially distributed random
time and then switches speed from $0$ to $1$.
It keeps this speed until it possibly
collides with car $i-1$.
At this moment, car $i$ is blocked, and remains so until car $i-1$ leaves this position.
After the departure of car $i-1$, car $i$
waits a new exponential time to depart.
This model was proposed in~\cite{CaceresFerrariPechersky07}, as a simplification of more complete
slow-to-start rules discussed in~\cite{GrayGriffeath01,GrigorescuKangSeppaelaeinen04,SopasakisKatsoulakis06} and references
therein.
Shneer and Stolyar~\cite{ShneerStolyar19} discuss stability issues of a discrete version of this process.
We prove rigorous results seemingly unattainable for the more complete
rules.

\paragraph{Informal description of results}

% The stopped cars at time $t$ (which have approximate density $\lambda-1$) condensate on \emph{traffic jams}, points of a process of approximate density $1/\sqrt t$. Each traffic jam has approximately $\sqrt t$ cars.

Saturated regime $\lambda>1$.
All the cars collide infinitely often and each car has asymptotic mean velocity $1/\lambda$.
The set of positions of moving cars converge to a Poisson process of parameter 1.
Any two fixed cars will be either both stopped at the same traffic jam or both moving, with probability tending to 1 in time.
The distances between consecutive traffic jams at time $t$, as well as and their
size, are of order $\sqrt t$.
In the rescaled process, each traffic jam corresponds to a double point at time
zero in the coalescing Brownian process of Arratia~\cite{Arratia79}, while the
size of the traffic jam is the distance at time $t$ between the
two Brownian motions starting at the double point.

Critical regime $\lambda =1$.
The set of positions of moving cars converge to a Poisson process of parameter 1.
Each car has asymptotic mean velocity $1$, but it collides infinitely often.
The system shows reminiscent condensation: the probability that a car is stopped is vanishing, but yet for any two fixed cars, the conditional probability that they are at the same traffic jam given that one
of them is stopped tends to $1$.
The distance between two traffic jams at time
$t$ is of order $t$ and their sizes are of order $\sqrt t$.
In the scaling limit, the evolution of the traffic jam configuration is not Markovian.

The unsaturated regime $\lambda<1$ has been considered in~\cite{CaceresFerrariPechersky07}, where it was shown that the configuration of moving cars converges to a Poisson Point Process with intensity $\lambda$.

\paragraph{The model}

The process is denoted by $(\pi,v)=((\pi(t),v(t));\,t\ge 0)$, where $\pi(t)=(\pi_i(t); i\in \Z)$ and $v(t)=(v_i(t);i\in \Z)$.
For each $i$, $\pi_i(t)\in\R$ represents the position of car $i$ at time $t$ and $v_i(t)\in\{0,1\}$ its speed.
The initial car positions are given by $\y_0 = \{y_i:i\in\Z\}\subseteq\R$, a Poisson process with intensity
$\lambda$ whose points are labeled so that $\cdots <y_{-2}<y_{-1}<y_0 \leqslant 0<y_1<y_2<
\cdots$.
The trajectories are defined by
\begin{equation}
\label{a10}
\begin{cases}
\pi_i(0)=y_i;\\
v_i(0)=0; \\ 
\dot\pi_i(t) = -v_i(t) \text{ for a.e. } t \\
v_i(t) $ jumps from $0$ to $1$ at rate 1 if $ \pi_i(t)>\pi_{i-1}(t)$;$ \\
v_i(t) = 0$ if $\pi_i(t)= \pi_{i-1}(t)
\end{cases}
\end{equation}
and the underlying probability measure will be denoted $P$.

\paragraph{Local behavior}

% We start describing the local behavior of cars in two cases: (a) as seen from a fixed position like the origin, and (b) as seen from a tagged car initially at the origin.

We start by describing the process as seen from a fixed observer.

Let $\y^1(t)=\{\pi_i(t):v_i(t)=1,\,i\in\Z\}$ be the set of positions
of the moving cars at time $t$ and let
$\y(t)=\{\pi_i(t):v_i(t)=0,\,i\in\Z\}$ be the set of positions of the
stopped cars, or \emph{traffic jams}.
Then the positions of the moving
cars converge to a Poisson process with a density $\min\{\lambda,1\}$,
while the places with traffic jams disappear.
% Let $\y(0)$ be a Poisson process of density $\lambda$ and $\y^1(0)=\emptyset$. Then,
\begin{prop}
\label{prop:satcond}
The set of traffic jams vanishes:
\[
\y(t) \ \stackrel{\rm a.s.}\longrightarrow \ \emptyset.
\]
Moreover,
\[
\y^1(t) \ \stackrel{\rm d}\longrightarrow \ {\rm Poisson}(\min\{1,\lambda\})
\]
as $t\to\infty$,
where $\stackrel{\rm d}\longrightarrow$ means convergence in distribution.
\end{prop}

% The proof if given in Section~\ref{sec:satcond}.

The next result concerns the behavior of individual cars.

\begin{prop}
\label{prop:speed}
The mean velocity of each car $i$ satisfies
\[
\frac{- \pi_i(t)}t \ { \mathop {\longrightarrow}^{\rm{a.s.}}_{ t\to\infty} }\
\min\{\tfrac{1}{\lambda},1\}.
\]
\end{prop}

\paragraph{Scaling limits}
We now consider the traffic jams and discuss the scaling limits in the saturated and critical regimes.
Recall $\y(t) $ is the set of traffic jams at time $t$.
Since all cars move at the same
speed, new traffic jams cannot appear and we have $\y(t')\subseteq\y(t)$ for
$t'>t$.
For the traffic jam $y\in\y(t)$ let $N_t(y):=\sum_i \I\{\pi_i(t)=y;
v_i(t)=0\}$ be the number of cars at $y$ at time $t$.
Then for each $t$,
$(\y_t,N_t)$ is a marked point process.

The scaling limit for $\lambda>1$ is a Markov process associated to a system of
coalescing Brownian motions with masses.
At each point $z\in\R$, start a
standard one-dimensional Brownian motion $(B^z_s)_{s\geqslant0}$, all motions
being independent of each other before they meet, and coalescing thereafter; see
Arratia~\cite{Arratia79} for a construction of the process.
At a given time
$t>0$, the set of positions $\z_t=\{z:B^{z+}_t > B^{z-}_t\}$ is a discrete
subset of $\R$.
Take $m_t:\z_t\to\R_+$, where $m_t(z)=B^{z+}_{t}-B^{z-}_{t}>0$
is the mass associated to site $z$.
For each fixed time $t$, $(\z_t,m_t)$ is a
marked point process.

\begin{theo}
\label{theo:1}
Let $\lambda>1$ and, for each $L>0$, rescale $(\y_t,N_t)$ by defining ${\y_t^L}
= \frac{\lambda-1}{\sqrt L} \y_{Lt}$ and $N_t^L(y)=\frac{1}{\sqrt L}
N_{Lt}\bigl(\frac{\sqrt L}{\lambda-1}\, y\bigr)$.
Then, as $L\to\infty$,
\begin{equation}
\label{eq:superrescaling}
\big({\y_t^L},N_t^L\big)_{t>0} \ \stackrel {\rm d} \longrightarrow\ 
\bigl(
\z_t,m_t
\bigr)_{t>0}
\end{equation}
in the sense of compact restrictions of finite-dimensional time projections.
\end{theo}

As a side remark, the evolution of $(\z_t,m_t)_{t>0}$ is Markovian.
Also, for each $t>0$, the point process $\z_t$ is not Poisson~\cite{Arratia79}.

We describe now the limiting process in the critical case $\lambda=1$.
Let
$(W_s)_{s\in\R}$ be a standard Brownian motion defined for all positive and
negative times $s$ and with $W_0=0$.
To each $x\in\R$ associate a new
Brownian motion $B^x=(B^x_s)_{s\geqslant x}$ starting at time $x$ at position
$B^x_x=W_x$ and being reflected by $W_s$ from above, so that $B^x_s \leqslant
W_s$ for all $s\geqslant x$.
Also let the motions $B^x_\cdot$ evolve independently until they coalesce, that
is, for $x'\leqslant x'' \leqslant s'\leqslant s''$, $B^{x'}_{s'} =
B^{x''}_{s'}$ implies $B^{x'}_{s''} = B^{x''}_{s''}$.
Coalescence and reflection imply that $B^{x'}_s \leqslant B^{x''}_s$
for $x' \leqslant x'' \leqslant s$.
So both limits $B^{y\pm}_s$ as
$x'\uparrow y$ or $x''\downarrow y$ exist.
For $t>0$ let
% $\x_t=\{y\in\R:B^{y+}_{y+t}>B^{y-}_{y+t}\}$
% and take $n_t:\x_t\to\R_+$
% as
\begin{equation}
\label{eq:masscontcrit}
\x_t=\{y\in\R:B^{y+}_{y+t}>B^{y-}_{y+t}\}
,
\quad
n_t(y)=B^{y+}_{y+t}-B^{y-}_{y+t}>0
,\
y \in \x_t
.
\end{equation}
Then $(\x_t,n_t)$ is a marked
point process for each $t>0$.
See Fig.~\ref{fig:critical-limit}.
From the picture the reader may notice that $(\x_t,n_t)_{t>0}$ is not Markovian.
We do not know whether $\x_t$ is Poisson.

\begin{figure}[tb]
\centering
\includegraphics[width=120mm]{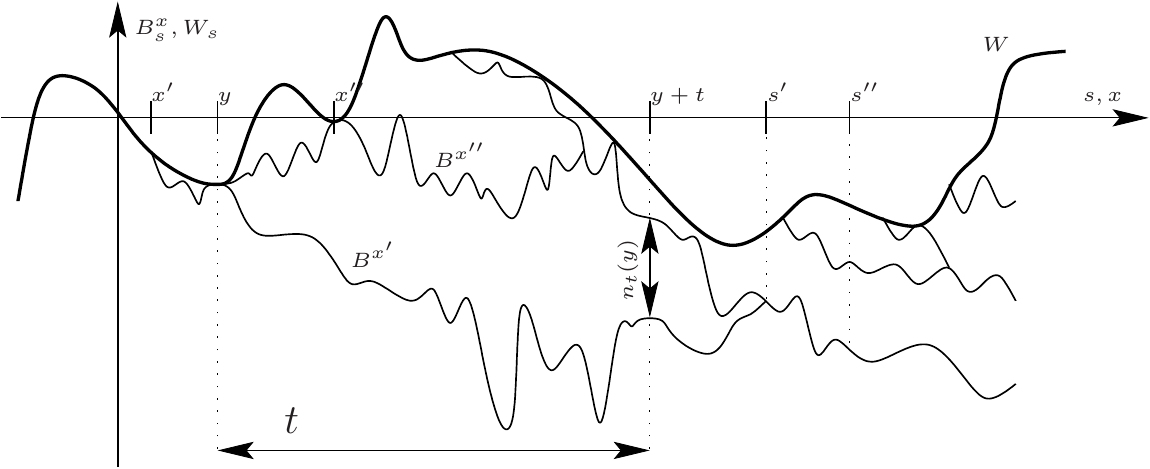}
\caption{Definition of $n_t(y)$ from coalescing reflected Brownian motions. Note that $n_s(y)>0$ for $s<s'$ and $n_s(y)=0$ for $s \geq s'$.}
\label{fig:critical-limit}
\end{figure}

\begin{theo}
\label{theo:scalingcritical}
Consider the critical case $\lambda=1$.
For each $L>0$, rescale $(\y_t,N_t)$ by defining
\begin{equation}
\label{eq:scaling}
{\y_t^L} = \frac{1}{L} \y_{Lt}
\quad
\text{ and }
\quad
N_t^L(y)=\frac{1}{\sqrt L} N_{Lt}(L \, y)
.
\end{equation}
Then, as $L\to\infty$,
\begin{equation}
\label{eq:sclimit}
\big({\y_t^L},N_t^L\big)_{t>0} \ \stackrel {\rm d} \longrightarrow\ 
(\x_t,n_t)_{t>0}
\end{equation}
in the sense of compact restrictions of finite-dimensional time projections.
\end{theo}

The main tool to prove the theorems is a representation of the motion using a
set of coalescing Poisson processes, described in \S\ref{sec:coupling}.

%\clearpage
\section{Coupling with coalescing reflected walks}
\label{sec:coupling}

In this section we describe an explicit construction for the slow-to-start traffic model, show how it can be coupled with a collection of coalescing reflected paths, and how to read the model off such a collection.

\subsection{Explicit construction}

Recall that $\y\subset \R$ is the Poisson process of parameter $\lambda$ where the
cars are located at time zero; $y_i\in\y$ is the position of car $i$.

For each $i\in\Z$, let $\d_i\subseteq\R^+$ be a Poisson process of parameter 1 independent of each other and of $\y$.
Label these random sets by $\d_i=\{d_{i,j}\}_{j\geq i}$ with $0<d_{i,i}<d_{i,i+1}<\cdots$.
The value of $d_{i,j}$ will be the departure time of car $j$ from position $y_i$, in case car $j$ happens to stop at $y_i$ (otherwise $d_{i,j}$ plays no role in the construction).

% Call $(d_{i,j},\,j\ge i)$ the epochs of the Poisson process $\a_i$.
% For convenience we start the labeling at $i$: $d_{i,i}$ is an exponential time of parameter $1$ and so are $d_{i,j+1}-d_{i,j}$ for $j\ge i$.
% Then car $i$ departs from $y_i$ at time $d_{i,i}$, while if car $j\ge i$ happens to be blocked at $y_i$, then it will depart at time $d_{i,j}$.

For each $j \ge i$, we will denote by $A(i,j)$ the arrival time of car $j$ to position $y_i$ and by $D(i,j)$ the departure time of car $j$ from position $y_i$.
When $D(i,j)=A(i,j)$, it means that car $j$ does not stop at $y_i$.
From $\y$ and $\d$, these times are defined recursively as follows.

At step $0$, define $A(i,i) = 0$ and $D(i,i) = d_{i,i}$, for all $i$.
After step $k$, $A(i,j)$ and $D(i,j)$ have been defined for every pair such that $i \le j \le i+k$.
At step $k+1$, for each $i\in\Z$ write $j=i+k+1$ and define
\begin{align}
\label{eq:defa}
A(i,j) &= D(i+1,j) + y_{i+1}-y_{i}
,
\\
\label{eq:defb}
D(i,j) &=
\begin{cases}
A(i,j), & A(i,j)>D(i,j-1), \\
d_{i,j}, & A(i,j)<D(i,j-1).
\end{cases}
\end{align}
From these times we define the process $(\pi,v)$ by
\begin{equation}
\label{a90}
(\pi_j(t),v_j(t)) =
\begin{cases}
(y_i,0), & A(i,j) \le t < D(i,j), \\
(y_i-t+D(i,j),1), & D(i,j) \le t < A(i-1,j).
\end{cases}
\end{equation}

The size of the traffic jam at $y_i$ at time $t$ is
given by
\[
N_i(t) = \sum_{j:j\ge i} \one\{A(i,j) \le t < D(i,j)\}
\]
and the set of traffic jams at time $t$ is given by
\[
\y(t) = \{y_i\in \y\,:\, N_i(t)>0\}
.
\]
Fig.~\ref{fig:graph41} illustrates this construction. Between times $A(1,3)$ and $D(1,1)$ there are 3 cars in the traffic jam at position $y_1$. 

\begin{remark*}
This construction was used in~\cite{CaceresFerrariPechersky07} to study the case $\lambda<1$. The main observation there was that from the picture shown in Fig.~\ref{fig:graph41}, one can read a stable $M/M/1$ process on the $x$-axis from left to right, see~\cite{CaceresFerrariPechersky07} for the details. This is the reason why our cars move from right to left.
\end{remark*}

\begin{figure}[tb]
 \hspace*{\fill}\includegraphics[trim= 0 0 1cm 1mm, clip,width=\textwidth]{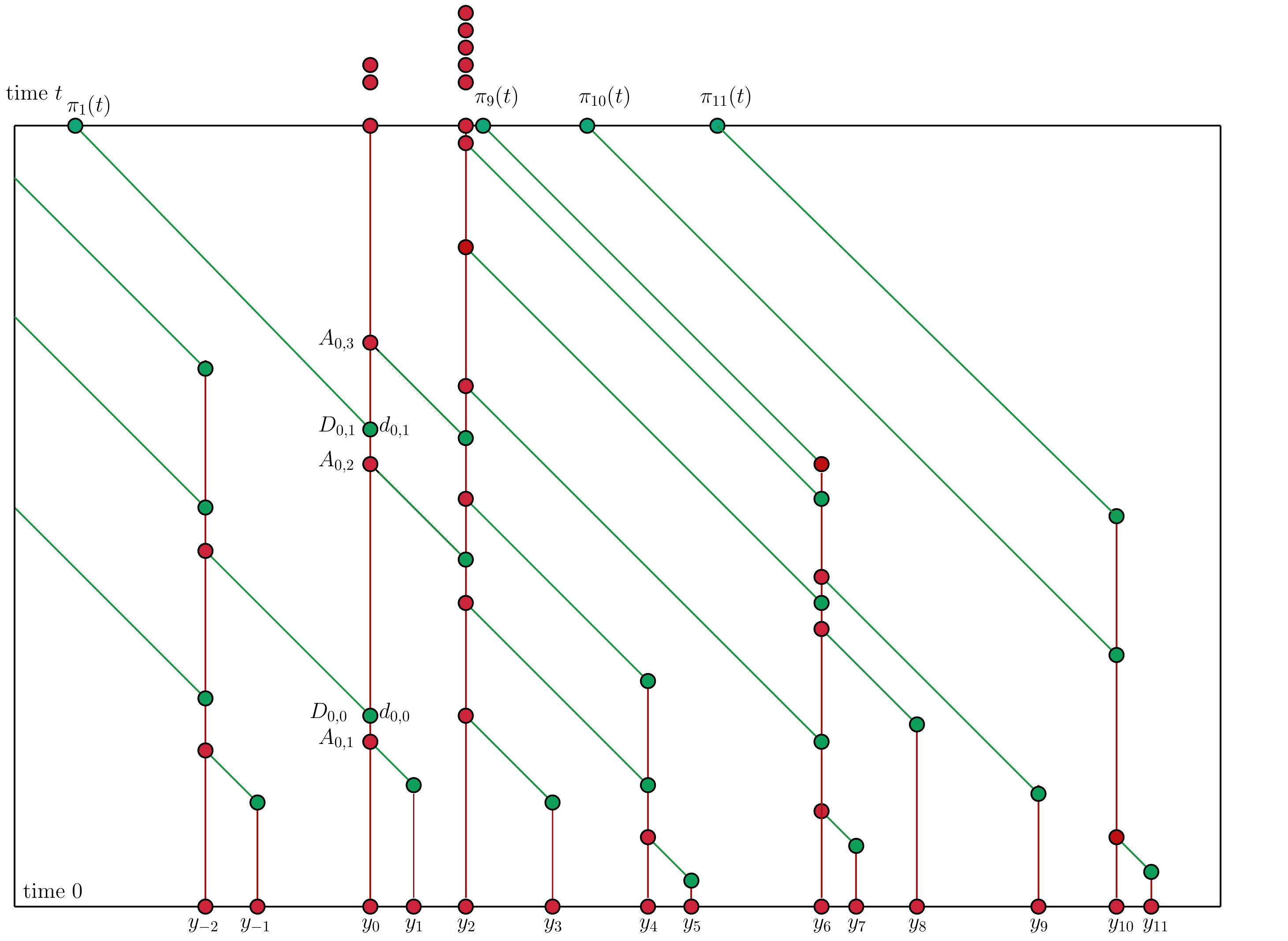}\hspace{\fill}
\caption{Graphical construction. Time is going up, cars are moving left. Red lines represent traffic jams, green lines are trajectories of moving cars. Red dots are either traffic jams at times 0 and $t$ or arrival times of new cars to traffic jams. Green dots represent departure times $D(i,j)=d_{i,j}$; we have not depicted those $d_{i,j}$ not becoming a car departure. At time $t$ there are traffic jams at $y_0$ with 2 stopped cars and at $y_2$ with 5 stopped cars while cars $1, 9, 10$ and $11$ are travellig at speed 1. Cars with label less than 1 are out of the picture at time $t$. }
\label{fig:graph41}
\end{figure} 
% \begin{figure}[tb]
% \hspace*{\fill}\includegraphics[width=\textwidth]{figures/graph-construction-I}\hspace{\fill}
% \caption{Graphical construction. Time is going up, cars are moving left.}
% \label{fig:graph-construction-I}
% \end{figure} 

% \clearpage
\subsection{Coalescing counting processes}
\label{sub:coalesce}

The Poisson process $\y$ induces a doubly infinite counting process
$(Y_x)_{x\in\R}$ defined by
\[
Y_x= i \text{ \ for \ } y_{i-1}\le x < y_i,
\] 
where we label $\y$ as $\y=\{y_i, \, i\in\Z\}$ with $y_0<0<y_1$ and $y_i<y_{i+1}$ for all $i$.
The process $Y$ jumps one unit up at each point $y_i\in\y$, so that $Y(y_i) = i+1$. This is the upper green path in the botton picture of Fig.~\ref{fig:graph-46}.

% Note that $Y(0)=1$
% , this definition just for.
% This process will be the upper boundary where a set of independent Poisson processes of parameter 1 will get reflected.

From $Y$ and $\d$, we define the counting processes $B^i=(B^i_x)_{x\in\R}$, $i\in\Z$, as follows.
First define
\begin{equation}
\label{eq:countbefore}
B^i_x = i \text{ for all } x < y_i
,
\quad
\text{ for all } i \in \Z
.
\end{equation}

Now let $i$ be fixed, and assume $B^{i+1}$ has been defined.
Let $B^i_x=i$ for $x \in [y_i,y_i + d_{i,i})$.
If $B^{i+1}_{y_i+d_{i,i}} = i+1$, let $B^i_x = B^{i+1}_x$ for all $x \geq y_i+d_{i,i}$, and the definition of $B^i$ is complete.
Otherwise define $B^i_x = i+1$ for $x \in [y_i+d_{i,i},y_i+d_{i,i+1})$.
If $B^{i+1}_{y_i+d_{i,i+1}} = i+2$, let $B^i_x = B^{i+1}_x$ for all $x \geq y_i+d_{i,i+1}$, and the definition of $B^i$ is complete.
Otherwise define $B^i_x = i+2$ for $x \in [y_i+d_{i,i+1},y_i+d_{i,i+2})$.
Continue this construction indefinitely.
In words, $B^i$ jumps by $+1$ at points $\{y_i+d_{i,j}\}_{j\geq i}$, but only until it meets $B^{i+1}$, and it follows $B^{i+1}$ after that (and $B^{i+1}$ is either jumping at points $\{y_{i+1}+d_{i+1,j}\}_{j\geq i+1}$ or following $B^{i+2}$, and so on).

To make this construction precise, we fix $x^* \in \R$ and restrict the above definition to $x < x^*$.
Once $x^*$ is fixed, there is $i^*$ such that $y_{i^*} > x^*$, whence $B^i$ is defined simply by~(\ref{eq:countbefore}) for all $i \ge i^*$.
Now $B^{i^*-1},B^{i^*-2},B^{i^*-3},\dots$ can be defined recursively following the rules just described.
Notice that this definition is consistent, i.e., for each $i$ and $x$ we obtain the same $B^i_x$ for all choices of $x^* > x$.

\begin{remark*}[The distribution of $Y$ and the $B^i$'s] 
The path $Y$ is a process that jumps at rate $1$.
Conditioned on $Y$, the $B^i$'s are distributed as a family of processes that start from $Y$, are reflected by $Y$ from above, and jump independently until they meet, after which they coalesce.
\end{remark*}

% \clearpage
\subsection{An equivalent construction}

% \begin{figure}[tb]
% \hspace*{\fill}\includegraphics[width=\textwidth]{figures/graph-construction-II}\hspace{\fill}
% \caption{Counting processes coupled with the graphical construction.
% This example corresponds to that of Fig.~\ref{fig:graph-construction-I}.
% We denote $\tilde d_{i,j}=y_i+d_{i,j}$. At time $t$ counted from $y_1$ there are 3 cars in the traffic jam located at $y_1$, so $N_1(t)=3$. We have $T(1,i) = \tilde d_{1,i}$ for $i=1,2,3$; $T(2,2) = \tilde d_{2,2}$, $T(2,3) = T(3,3) = \tilde d_{3,3}$, indeed $B^2$ and $B^3$ coalesce at time~$\tilde d_{2,2}$.}
% \label{fig:graph-construction-II}
% \end{figure} 
\begin{figure}[p]
\hspace*{\fill}\includegraphics[ trim= 0 0 0 0cm,clip, width=\textwidth]{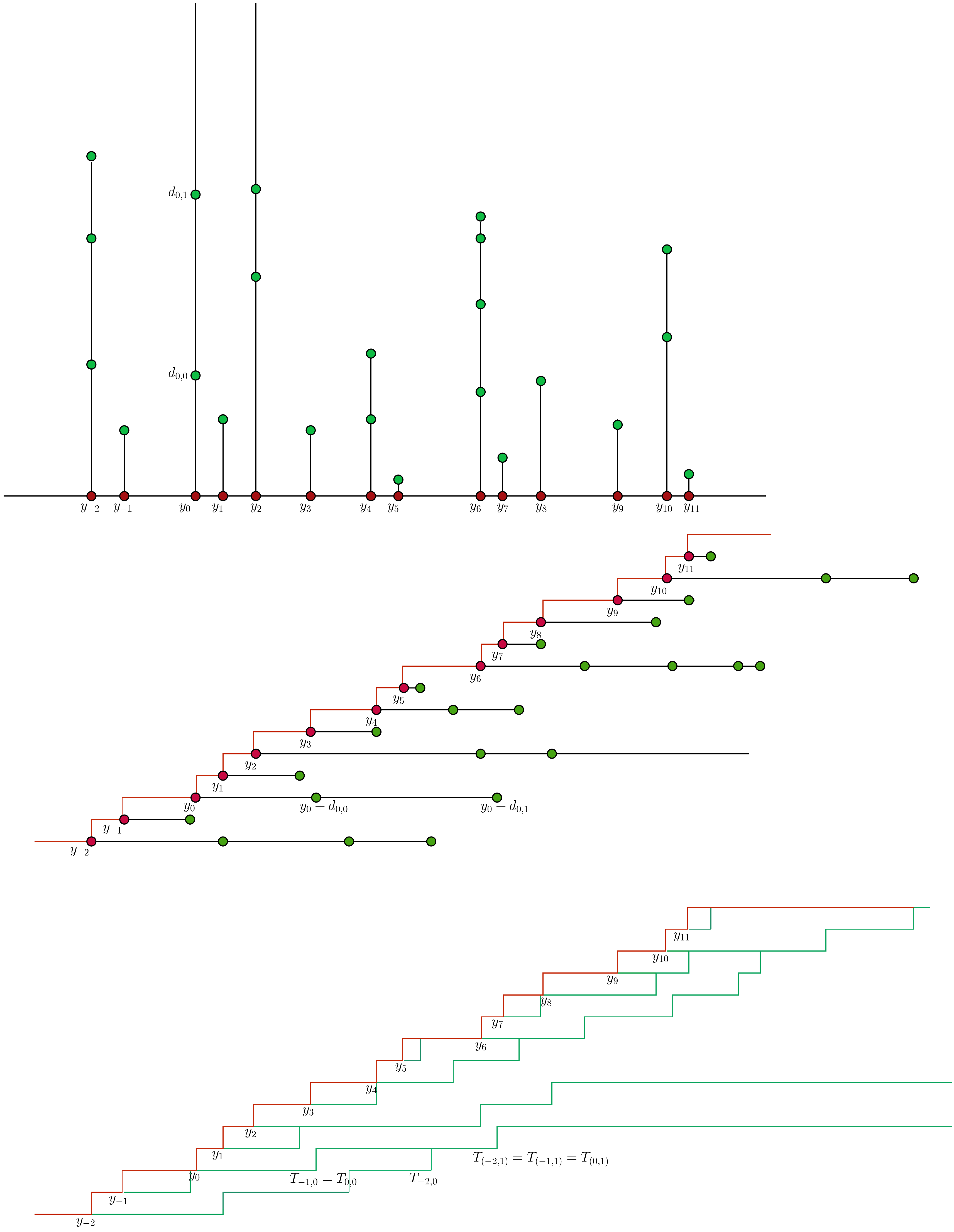}\hspace{\fill}
\caption{Coalescing random walks produced by the graphical construction of Fig.~\ref{fig:graph41}. The Poisson process $(d_{i,j})_{j\ge i}$ is depicted vertically in the upper figure and horizontally as $(y_i+d_{i,j})_{j\ge i}$ in the center figure; we have depicted only those $d_{i,j}$ used by $B^i$ before coalescing with $B^{i+1}$. The bottom figure uses $(y_i+d_{i,j})_{j\ge i}$ as counting coalescing processes. }
\label{fig:graph-46}
\end{figure} 

We now construct a version of the slow-to-start process $(\pi,v)$ in function of the coalescing reflected counting processes $(Y,B)$.

For $i\in\Z$ define $T(i,i-1)=y_i$ and for $j \ge i$ let $T(i,j)$ be the point when walk $i$ leaves level $j$:
\begin{equation}
\label{eqtij}
T(i,j):= \inf\{x \ge y_i \,:\, B^i_x=j+1\}
\end{equation}
In this way $T=(T(i,j))_{i\in\Z, j \ge i-1}$ is determined by $(Y,B)$.
Note that $B^i_x \le B^{i+1}_x$ for all $i$ and $x$, so $T(i+1,j) \le T(i,j) < T(i,j+1)$.

We define the trajectories of the cars in function of $T$ as follows.
For each
$j$, the initial position of car $j$ is $\pi_j(0)=y_j$.
For $i \le j$, car $j$ will arrive to $y_i$ at time $T(i+1,j)-y_i$ and leave $y_i$ at time
$T(i,j)-y_i$ (in the meantime, the speed of car $i$ is zero).
During the time interval $[T(i,j)-y_i,T(i,j)-y_{i-1})$ car $j$ travels at speed 1.
More precisely,
\begin{equation}
\label{eq:procfromjumps}
(\pi_j(t),v_j(t))=
\begin{cases}
(y_i,0)
,&
T(i+1,j)-y_i \le t < T(i,j)-y_i
,\\
(T(i,j) - t,1)
,&
T(i,j)-y_i \le t < T(i,j)-y_{i-1}
.
\end{cases}
\end{equation}
In case $T(i,j)=T(j+1,i)$, car $j$ does not stop at $y_i$. Fig.~\ref{fig:graph41} describes the construction of the coalescing random walks corresponding to the cars of Fig.~\ref{fig:graph-46}.

\begin{prop}
\label{a30}
The slow-to-start process $(\pi,v)$ defined in~(\ref{eq:procfromjumps}) is the same as the one defined by~\eqref{a90} almost surely.
The size of the traffic jam at $y_i$ at
 time $t$ is given by
\begin{equation}
\label{eqjams}
N_{y_i}(t) =
B^{i+1}_{y_i+t} - B^{i}_{y_i+t}
.
\end{equation}
\end{prop}
\begin{proof}
To show agreement between~\eqref{a90} and~\eqref{eq:procfromjumps},
we will show for all $i\le j$
\begin{align}
A(i,j) &= T(i+1,j)-y_i \label{a66}\\
D(i,j) &= T(i,j)-y_i \label{a67}\\
T(i,j)-t &= y_i-t+D(i,j) \quad \hbox{ for } \quad D(i,j) \le t < A(i-1,j) \label{a68}
,
\end{align}
where $A$ and $D$ are given by~\eqref{eq:defa}-\eqref{eq:defb} in terms of $\y$ and $\d$, and $(T(i,j))_{i,j}$ is given by~\eqref{eqtij} in terms of $Y$ and $(B^i)_{i}$, which in turn are constructed from $\y$ and $\d$.

We will prove~\eqref{a66} and~\eqref{a67} simultaneously by induction on $j-i \ge 0$, and~\eqref{a68} follows from~\eqref{a67}.
For $j=i$, $T(i+1,i)= y_i$ and $T(i,j)= y_i+ d_{i,i}$ so~\eqref{a66} and~\eqref{a67} hold.
Before proceeding, note that from definition of $(B^i)_{i\in\Z}$ and $(T(i,j))_{i\in\Z, j\ge i-1}$, we have
\begin{equation}
\label{eq:recursivetij}
T(i,j) =
\begin{cases}
T(i+1,j) ,&
T(i,j-1) < T(i+1,j)
,
\\
y_i + d_{i,j}
,&
\text{otherwise}
.
\end{cases}
\end{equation}
Now let $i\in\Z$, $j\ge i$, and suppose~\eqref{a66} and~\eqref{a67} hold for all pairs $i',j'$ such that $0 \le j'-i' < j-i$.
From~\eqref{eq:defa}, by direct substitution of~\eqref{a67} with $i+1$ instead if $i$, we get~\eqref{a66}.
From~\eqref{eq:recursivetij}, by substituting~\eqref{a67} with $j-1$ instead of $j$ and~\eqref{a66}, we get
\begin{equation}
\nonumber
T(i,j) - y_j =
\begin{cases}
A(i,j) ,&
D(i,j-1)+y_i < A(i,j)+y_i
,
\\
d_{i,j}
,&
\text{otherwise}
,
\end{cases}
\end{equation}
so $T(i,j) - y_j = D(i,j)$, proving~\eqref{a67}.

Finally, from~(\ref{eqtij})~and~(\ref{eq:procfromjumps}),
\begin{multline}
% \label{jampop}
\nonumber
\big\{j:\pi_j(t)=y_i,v_j(t)=0\big\}
=
\bigl\{j: T(i+1,j) \leqslant t+y_i < T(i,j) \bigr\},
=
\\
=
\bigl\{j:B^i_{t+y_i} \leqslant j < B^{i+1}_{t+y_i} \bigr\},
\end{multline}
so~(\ref{eqjams}) holds.
\end{proof}

%\clearpage
\section{Saturation and condensation}
\label{sec:satcond}

In this section we prove
Proposition~\ref{prop:satcond}

%NEED TO IMPROVE THIS SKETCH
Consider the model adding a traffic jam initially at $0$ denoted $y_0=0$. The traffic jam at $y_1$ disappears at the minimal time $x$ such that $B^1_x=B^2_x$, that is, when the walks starting at $y_1$ and $y_2$ coalesce. This time is dominated from above by the minimal $x'>0$ such that $B^1_{x'}= Y_{x'}$; a time with finite expectation, because $Y$ jumps at rate $\lambda<1$, the jump rate of $B^1$.

The law of the car configuration with the addition of a car at zero is the Palm measure. Since under the Palm measure all cars have the same law, we conclude that a.s.\ every traffic jam disappears. This shows the first part of the proposition.

For the second part, we recall the observation from~\cite{CaceresFerrariPechersky07} that the set of times when a car crosses the origin is distributed as the set of departure times of a $M/M/1$ system with arrival rate $\lambda$ and service rate $1$, which converges to $\mathrm{Poisson}(\min\{\lambda,1\})$.
To prove convergence of $\y^1$, we need to show that $\y^1 \cap [-K,K]$ converges to $\mathrm{Poisson}(\min\{\lambda,1\})$ on $[-K,K]$ for every fixed $K$.
By translation invariance, we can consider $\y^1 \cap [-2K,0]$ instead.
By the first part, with high probability as $t\to\infty$, there are no traffic jams on $[-2K,0]$ at time $t-2K$.
On this event, the set $\y^1(t) \cap [-2K,0]$ equals the set times $s\in[t-2K,t]$ when a car crosses the origin, shifted by $-t$.
The convergence in distribution then follows.

%\clearpage
\section{Speed of individual cars}
\label{sec:car}

In this section we prove Proposition~\ref{prop:speed}.

It suffices to consider the Palm distribution and look at the speed of the $0$-th car starting at $y_0=0$. The position of this car is tracked along the line $x=0$ in the bottom picture of Fig.~\ref{fig:graph-46}. By~\eqref{eq:procfromjumps}, when $t\in [T(i+1,0)-y_i, T(i,0)-y_i)$, car 0 is at position $y_i$. Hence, 
\begin{align}
 \label{p17}
v=\lim_{t\to\infty} \frac{-\pi_t(0)}{t} = \lim_{n\to \infty} \frac{-y_{-n}}{T(n,0)-y_{-n}}.
\end{align}
given that the limit exists. Since $y_{-n}=\sup\{t<0:|\y\cap[t,0)|=n\}$ is the first time that a Poisson process of rate $\lambda$ has $n$ events, $y_{-n}/n$ converges to $-1/\lambda$.

We have 3 cases, depicted macroscopically in Fig.~\ref{fig:propo-12}.
\begin{figure}[tb]
\centering
\includegraphics[trim = 0 0 3cm 0, clip,width=\textwidth]{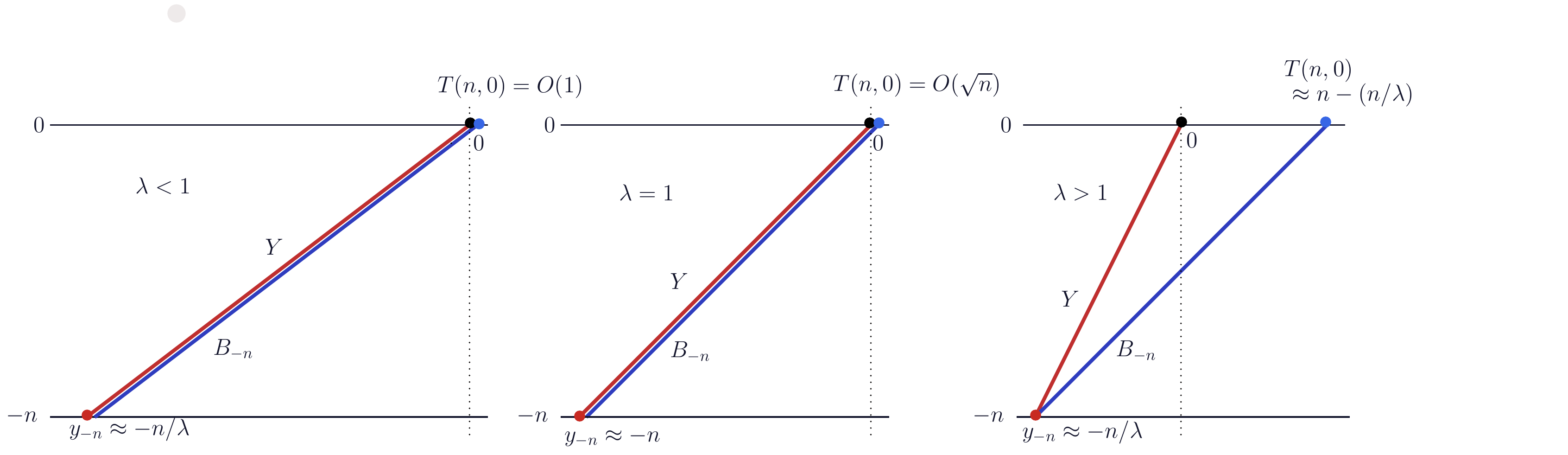}
\caption{Proof of Proposition~\ref{prop:speed}. The red lines describe the macroscopic behavior of $Y$ between times $y_{-n}$ and $0$. Blue lines correspond to $B^{-n}$ in the same time interval. Red dot is $y_{-n}$. Blue dot is $T(n,0)$. Black dot is space-time origin; time is going up. In the subcritical and critical cases the blue and red lines coincide in this scale; we have separated them to indicate that both lines are in the same position. }
\label{fig:propo-12}
\end{figure}
When $\lambda<1$, $T(n,0)$ is order 1. Indeed this is the first time that the walk $B^{-n}$, which starts at $y_{-n}$, leaves the level 0. But, $B^{-n}$ jumps at rate 1 up and it is reflected by the walk $Y$, which jumps at rate $\lambda<1$, the difference $B^{-n}_t-Y_t$ is of order one, and it is stochastically dominated by a geometric random variable of parameter $\lambda$ (the stationary distribution for $B^{-n}_t-Y_t$). Hence, $v=1$ when $\lambda<1$.

The critical case $\lambda=1$ uses the same argument but now both $Y$ and $B^{-n}$ jump at rate 1. Hence, their distance at the hitting time of $n$ is of order $\sqrt{n}$ and so is $T(n,0)$. On the other hand, $y_{-n}$ is of order $-n$ in this case and we have that $v=1$ when $\lambda=1$.

In the supercritical case $\lambda>1$ we have
\begin{align}
 \label{p18}
 \frac{y_{-n}}{n} \to_n -1/\lambda;\quad\text{and}\quad\frac{n}{T(n,0)-y_{-n}}=\frac{B^{-n}_{T(n,0)}-B^{-n}_{y_{-n}}}{T(n,0)-y_{-n}} \to_n 1.
\end{align}
Substituting in~\eqref{p17} we get $v=\frac1\lambda $ when $\lambda>1$.
% \begin{align}
% \lim_{t\to\infty} \frac{\pi_t(0)}{t} =
% \end{align}
% Hence in this case $T(n,0)$ is of order $n-\frac{n}{\lambda}$

% for any $i< 0$. $T(i,0)-y_i$ is a sum of $-i$ independent random variables with mean 1. Hence,
% \begin{align}
% \lim_{t\to\infty} \frac{\pi_t(0)}{t} = \lim_{n\to \infty} \frac{y_n}{n} = 
% \end{align}
% $\frac{\pi_t(0)}{t}$ converges to $\frac{y_i}{n}$

% Let $i^\pm = - (n \pm \varepsilon n)$.
% By LLN (or Chernoff to get probabilities summable), with high probability $T(i^+ +1,0) > n$, hence at time $t=n$ the tagged car is to the right of $y_{i^+}$.
% Likewise, with high probability $T(i^-,0) < n$ hence at time $t=n$ the tagged car is to the left of $y_{i^-}$.

% Use Borel-Cantelli, etc.

%\clearpage
\section{Scaling limits}

The proof of Theorems~\ref{theo:1}~and~\ref{theo:scalingcritical} goes as follows.

First, as seen in the previous section, the configuration and size of traffic jams $(\y_t,N_t)$ is given in terms of differences in the height of coalescing paths $\{B^i\}_{i\in\Z}$ that are reflected by another path $Y$.
By subtracting the drift and rescaling this collection of paths diffusively, in the scaling limit we obtain a collection $\mathcal{W}$ of coalescing-reflected Brownian motions instead of jump processes. Now notice formula~\eqref{eq:masscontcrit} is the continuous analogue of~\eqref{eqjams}.

Second, although this explains what the scaling limits should be, we still have to specify a topology so the claim that $(\y^L_t,N^L_t)$ converges to $(\x_t,n_t)$ can have a precise meaning, even for fixed $t$.

Third, if we are going to use the fact that the rescaled collection $\W^L$ converges in distribution to $\mathcal{W}$, then we need to specify a topology on collections of paths as well.
This topology should be weak enough so that we can prove that $\W^L$ converges in distribution to $\mathcal{W}$, but strong enough so that the map $\mathcal{W} \mapsto (\x_t,n_t)$ given by~\eqref{eq:masscontcrit} is continuous at a.e.\ $\mathcal{W}$.
These facts all together will imply that $(\y^L_t,N^L_t) \overset{\mathrm{d}}{\to} (\x_t,n_t)$.

For collections of paths, we will use same topology as the one introduced in~\cite{FontesIsopiNewmanRavishankar04} for the Brownian web.
This enables the use of existing results about convergence of a collection of jump paths to the latter process.
So besides checking that the scalings stated in Theorems~\ref{theo:1}~and~\ref{theo:scalingcritical} are the correct ones, we are left with the task of proving that the map $\mathcal{W} \mapsto (\x_t,n_t)$ given by~\eqref{eq:masscontcrit} is continuous at a.e.\ realization of the Brownian web $\mathcal{W}$.

% There are still some complications that we will avoid: we only prove convergence of $(\y^L_t,N^L_t)_{t=t_1,t_2,\dots,t_k}$ for arbitrary finite sets $\{t_1,\dots,t_k\}\subseteq \R_+$ rather than convergence of $(\y^L_t,N^L_t)_{t > 0}$ with Skorohod topologies etc, and we make no attempts to consider the strongest possible topology in the space of configurations $(\x_t,n_t)$.
% Strengthening the topology would introduce yet more technicalities.

% \clearpage
\subsection{Finding the scaling limits}

Suppose $\lambda>1$.
Consider the path $(Y_x)_{x\in \R}$ and the collection of paths $\{(B^i_s)_{s \ge y_i} : i \in \Z\}$ constructed in \S\ref{sub:coalesce}.
The paths $B^i$ jump by $+1$ unless they are blocked by $Y$.
If we consider $(Y_x-x)_{x\in\R}$ and $(B^i_s - s)_{s \ge y_i}$, we obtain a new curve $\tilde{Y}$ that has mean drift $\lambda-1$ and new curves $\tilde{B}^i$ that have zero mean drift.
Now rescaling space by a large $L>1$ and the paths $\tilde{Y}$ and $\tilde{B}$ by $\sqrt{L}$, we get a new path $Y^L$ with a large slope $(\lambda-1)\sqrt{L}$, and a path collection $\B^L$ which approximates coalescing Brownian motions started at $(x,Y^L_x)$, each one reflected by $Y^L$ during a microscopic amount of time, see Fig.~\ref{fig:scaling-supercritical}.

\begin{figure}[tb]
\hspace*{\fill}\includegraphics[width=.95\textwidth]{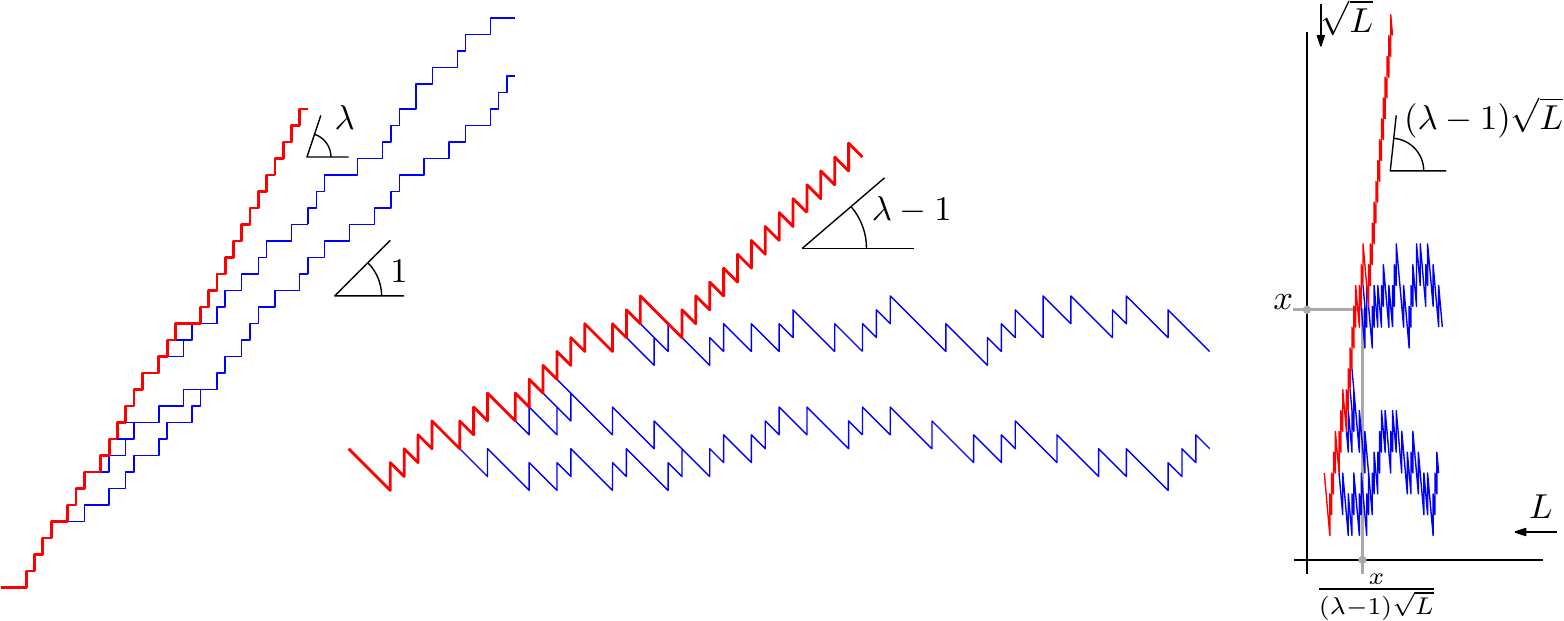}\hspace{\fill}
\caption{Scaling for $\lambda>1$ and large $L$. On the left, paths $Y$ in red and $B$ in blue; on the center, same paths minus the blue drift; on the right, diffusive rescaling of these paths by a large $L$.}
\label{fig:scaling-supercritical}
\end{figure}

Fix $x\in\R$ and note that the height $x$ is attained by $Y^L$ at point $\frac{x+o(1)}{(\lambda-1)\sqrt{L}}$, simply by law of large numbers. At the microscopic scale, it corresponds to position $\frac{\sqrt{L}}{\lambda-1} \, x + o(\sqrt{L})$.
In the limit $L\to \infty$, the path $Y^L$ becomes vertical and paths $B$ converge to coalescing Brownian motions started at points $(0,x)$.

The relation between height $x$ and microscopic position $\frac{\sqrt{L}}{\lambda-1} x$ holds in the limit, although it cannot be visualized under the diffusive scale (this happens because the horizontal axis mixes space and time, and the scale for spatial location of traffic jams is $\frac{\sqrt{L}}{\lambda-1}$ whereas the scale for time is $L$).
This explains the scaling of $\y_t$ in~\eqref{eq:superrescaling}.

\bigskip

We now consider $\lambda=1$.
Since the curve $Y$ has asymptotic slope $1$, after subtracting the drift and applying a diffusive scaling, it converges to a Brownian motion $(W_x)_{x\in\R}$ rather than a vertical line, see Fig.~\ref{fig:scaling-critical}.

\begin{figure}[tb]
\hspace*{\fill}\includegraphics[width=.75\textwidth]{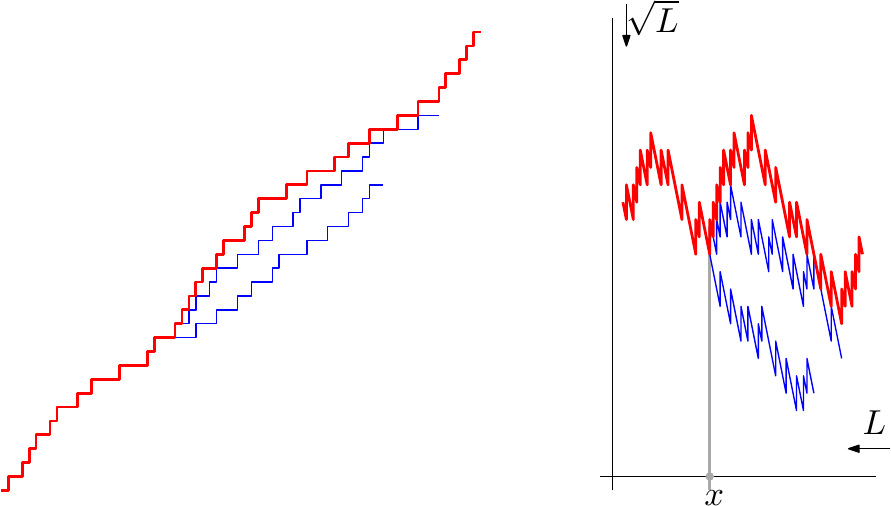}\hspace{\fill}
\caption{Scaling for $\lambda=1$ and large $L$. On the left, paths $Y$ in red and $B$ in blue; on the right, diffusive rescaling of these paths by a large $L$.}
\label{fig:scaling-critical}
\end{figure}

The limiting process is a little more complicated. The collections of paths $B$ converges to a collection of coalescing Brownian motions started from $(x,W_x)$ for each $x\in\R$, reflected by $W$ from above, rather than starting from the vertical axis without reflection.
The macroscopic location of a traffic jam, which is identified by the existence of nearby curves $B^{x-}$ and $B^{x+}$ that do not coalesce during $t$ time units, equals $x$. Hence, unlike the supercritical case, the location of a traffic jam can be visualized in the diffusively rescaled picture. Such a point corresponds to the microscopic position $Lx$, which explains the scaling of $\y_t$ in~\eqref{eq:scaling}.

In the remainder of this section we prove Theorem~\ref{theo:scalingcritical}.
Because of the vertical collapse of the red line in Fig.~\ref{fig:scaling-supercritical}, the proof of Theorem~\ref{theo:1} is tedious and technically simpler, so we omit it.

% \clearpage
\subsection{Brownian web and path collections}

Let us first describe how to read the process $(\x_t,n_t)_{t>0}$ off a double Brownian web $\W$ (see~\cite{FontesIsopiNewmanRavishankar04} for its definition).
The paths $(B^x_s)_{s \ge x}$ described at the introduction should be reflected by a Brownian motion $(W_x)_{x \in \R}$.
Reflection of paths are not a feature of the Brownian web, but \emph{primal paths} (those which move from left to right while diffuse vertically) do get reflected by \emph{dual paths} (those which move form right to left).
So we restrict the paths to $x \le s \le 0$.

We want to study the point process
\begin{equation}
\label{eq:configuration}
\cP
=
\big\{
(t,y,m) : y \in \x_{t}, m = n_{t}(y)
\big\}
.
\end{equation}
More precisely, given $0<\varepsilon<M<\infty$ and $\t = \{t_1,\dots,t_n\} \subseteq (\varepsilon,M]$, we want to describe the projection
$\cP_{\varepsilon,M,\t}:=\cP \cap \t \times [-4M,-2M] \times [\varepsilon,\infty)$.
By translation invariance, restricting space to $y \in [-4M,-2M]$ is equivalent to looking at $[-M,M]$.

We say that a collection $\B$ of semi-infinite paths in the $xu$-plane is \emph{good} if (i) for each point $(x,u)$ at least one primal path and one dual path start at $(x,u)$, (ii) paths in the primal do not cross paths in the dual nor other paths in the primal, (iii) for each path in the primal, every sub-path obtained by starting at a later time is also in $\B$, and (iv) $\B$ is closed in the topology of~\cite{FontesIsopiNewmanRavishankar04}.
We note that the double Brownian web $\W$ is a.s.\ good.

\begin{figure}[tb]
\hspace*{\fill}\includegraphics[page=1,width=.9\textwidth]{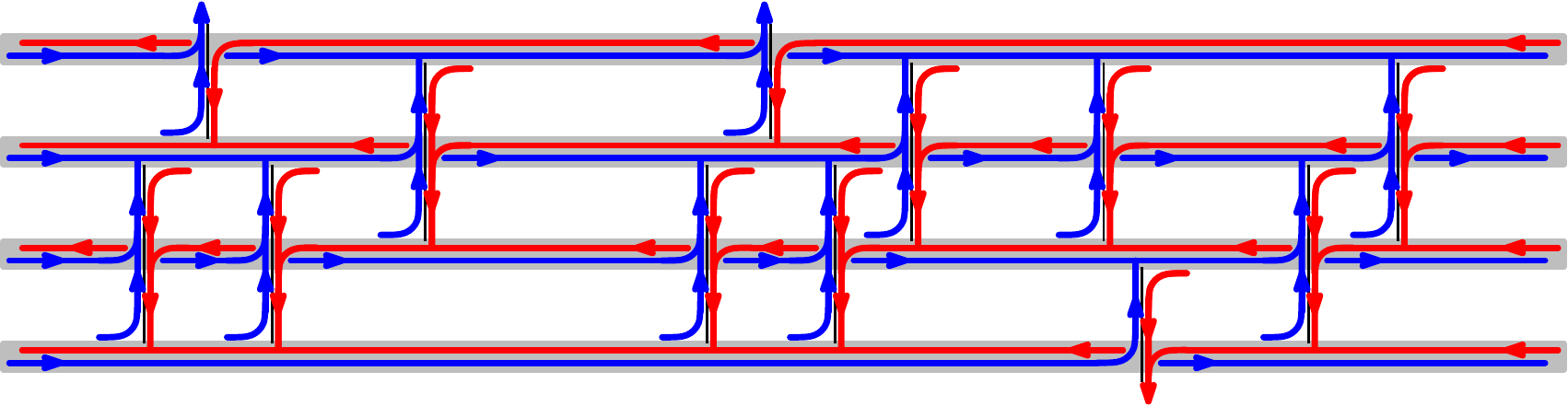}\hspace{\fill}
\caption{Construction of a good path collection $\B$.}
\label{fig:primaldual}
\end{figure}

We define the map $\Psi$ that assigns to a good set $\B$ of paths a set $\Psi(\B)$ of triples of the form~\eqref{eq:configuration} as follows.
At $x=0$, follow a dual path $(Y_x)_{x \le 0} \in \B$ started at $Y_0=0$.
(Normally there will be a single dual path at $(0,0)$, otherwise just take the lowest one.)
Now for each $x \in [-5M,-M]$ let $(B^x_s)$ be the lowest path started at $(x,Y_x)$.
Since $\B$ is closed and paths do not cross, $B^x$ is the monotone limit of paths starting at $(x,u)$ as ${u \uparrow Y_x}$, and therefore $B^x_s \leqslant Y_s$ for all $s\in[x,0]$.
Finally, $\Psi(\B)$ is defined as the set of all triples $(t,y,m)$ with $t \in\t$, $y \in [-4M,-2M]$, $m\in[\varepsilon,\infty)$ and satisfying~\eqref{eq:masscontcrit}.

In the above construction, in case $\B$ is the double Brownian web $\W$, the path $Y$ corresponds to the Brownian motion $W$ restricted to $x \le 0$.

Regarding the scaling limit~(\ref{eq:sclimit}), we want to show that $\cP^L \stackrel{\rm d}\to \cP$ as $L\to\infty$, where
\[
\cP^L
=
\big\{
(t,y,m) : Ly \in \y_{Lt}, \sqrt{L}\, m = N_{Lt}(Ly)
\big\}
.
\]
The above convergence is in the sense of compact restrictions of finite-dimensional time projections, as stated in Theorems~\ref{theo:1}~and~\ref{theo:scalingcritical}.
More precisely, for every $0<\varepsilon<M<\infty$ and $\t = \{t_1,\dots,t_n\} \subseteq (\varepsilon,M]$,
the projection
$\cP^L_{\varepsilon,M,\t}$
converges as a random counting measure to $\cP_{\varepsilon,M,\t}$.

We now describe another collection of paths, illustrated in Fig.~\ref{fig:primaldual}.
On the $xu$-plane, for each line with integer height $i\in \Z$, consider a Poisson Point Process with intensity 1.
At each point $(x,i)$ we start a primal path which jumps up by $+1$ each time it encounters a Poisson mark.
%We also start a dual path which jumps down by $-1$ at the marks, as in Fig.~\ref{fig:primaldual}.

For each $x \in \R$ and $u \in (i-1,i]$, start also a primal path which stays constant until the first Poisson mark of level $i-1$ or $i$. At this point, let it jump to $i$ or $i+1$, respectively, and follow the previous rules after that.

In this construction, primal paths coalesce as soon as they meet.
They may overlap with dual paths but not cross them.
Finally, taking the closure of this collection, we get a good set $\B$ of paths.

\begin{figure}[tb]
\hspace*{\fill}\includegraphics[page=2,width=.9\textwidth]{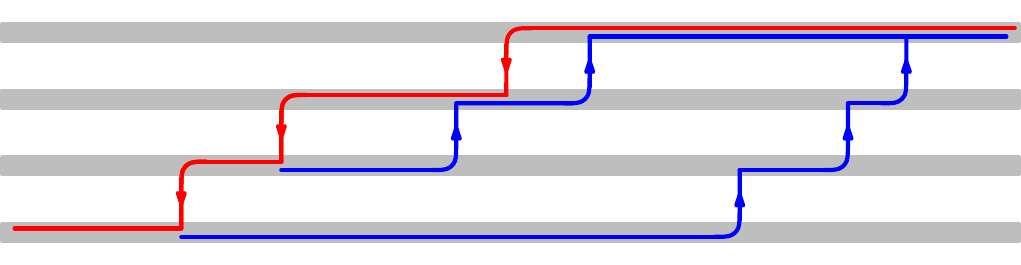}\hspace{\fill}
\caption{Reading $(Y_x)_{x \le 0}$ and $(B^i_s)_{y_i \le s \le 0}$ off the path collection of Fig.~\ref{fig:primaldual}. This picture shows $y_0,y_{-1},y_{-2}$ and part of $B^{0},B^{-1},B^{-2}$ from the example of Fig.~\ref{fig:graph-46}.}
\label{fig:reading}
\end{figure}

We now describe how to read the processes $(Y_x)_{x \le 0}$ and $(B^i_s)_{y_i \le s \le 0}$ introduced in \S\ref{sub:coalesce} off this collection $\B$, as shown in Fig.~\ref{fig:reading}.
To get the curve $(Y_x)_{x \le 0}$, follow the dual path started at $(x=0,i=1)$.
Then for $i\le 0$, the point $y_i = \sup\{x:Y_x \leq i\}$ is a ``backward hitting time'' of level $i$.
To get the curve $(B^i_s)_{y_i \le s \le 0}$, follow the primal path that starts at $(y_i,i)$.
These paths have the same distribution as those described in \S\ref{sub:coalesce}.
As a technical point, we add paths at $\infty$ to make this collection $\B$ compact in the~\cite{FontesIsopiNewmanRavishankar04} topology.

To conclude, we note that the same map $\Psi$ described above will now coincide with~\eqref{eqjams}.
To get the rescaled process $\cP^L_{\varepsilon,M,\t}$ we need to take $\Psi(\B^L)$, with
$$\B^L = \Big\{ \big( \tfrac{B_{Lx}-Lx}{\sqrt{L}}\big)_x : B \in \B \Big\}.$$
Since $\B^L \stackrel{\rm d}\to \W$, it only remains to show that $\Psi$ is continuous at a.e.\ $\W$.

%\clearpage
\subsection{Continuity almost everywhere}
\label{sub:continuous}

In this section we prove the following lemma.

\begin{lemma}
\label{lemma:continuous}
For each $\varepsilon,M,\t$ fixed, the map $\Psi$ is continuous at a.e.\ realization $\W$ of the Brownian web.
\end{lemma}

The notion of continuity is related to the Brownian web metric $d(\cdot,\cdot)$ as defined in~\cite{FontesIsopiNewmanRavishankar04}.
Recalling~\eqref{eq:masscontcrit} and the definition of $\Psi$, for fixed $\varepsilon,M,\t$, the event $(t,y,m) \in \Psi(\W)$ means that $m > \varepsilon$, $-4M < y < -2M$, $y \in \x_t$, and $n_t(y)=m$ (except for the zero-probability event that $-4M \in \x_t$, $-2M \in \x_t$ or $n(y)=\varepsilon$ for some $y \in \x_t$).
So Lemma~\ref{lemma:continuous} is equivalent to the following.

\begin{lemma}
\label{lemma:continuity}
For each $\varepsilon,M,\t$ fixed, for a.e.\ $\W$, for every $(t,y,m) \in \Psi(\W)$ and $\varepsilon>0$, there exists $\delta>0$ such that, for every 
good collection $\W'$ of paths with $d(\W,\W')<\delta$, there exists unique $y'$ and $m'$ such that $|y-y'|<\varepsilon$, $|m-m'|<\varepsilon$, and $(t,y',m')\in \Psi(\W')$.
%For a.e.\ $\W$, for each
%$y\in\x(t)$ and $\delta>0$ there is $\varepsilon>0$ such that for $\W'$ a compact
%collection of paths with $d(\W',\W)<\varepsilon$ there is a unique $y'\in\x(t;\W')$
%such that $\max|y-y'|,|n(y';\W')-N|<\delta$, where $N=n(y;\W)$.
\end{lemma}

In the remainder of this section we prove Lemma~\ref{lemma:continuity}.
The proof consists in a tedious analysis of properties of $\W$ that hold on the event $(t,y,m) \in \Psi(\W)$.

We will show existence of small (width and height less than $2\varepsilon$) rectangular boxes $C_1,C_2,C_3,D_1,D_2,D_3,D_4$ with the following properties.
Box $D_1$ is horizontally centered at $y$.
Boxes $D_2$ and $D_3$ are contained in $D_1$.
Box $D_4$ is centered at the origin.
Boxes $C_1,C_2,C_3$ are horizontally centered at $y+t$, and $C_3$ is above $C_2$ which is $m$ units above $C_1$.
All dual paths of $\W$ starting in $D_4$ coalesce at some time $s \in [y+t,0]$, they pass between $C_2$ and $C_3$, and they cross $D_1$, $D_2$ and $D_3$ horizontally.
Paths of $\W$ starting in $D_1$ cross either $C_1$, $C_2$ or $C_3$ horizontally.
Paths of $\W$ starting in $D_2$ cross either $C_1$ or $C_2$, and both cases occur.
Paths of $\W$ starting in $D_3$ cross either $C_2$ or $C_3$, and both cases occur.
See Fig.~\ref{fig:continuity}.

\begin{figure}[tb]
\centering
\includegraphics[page=1,width=\textwidth]{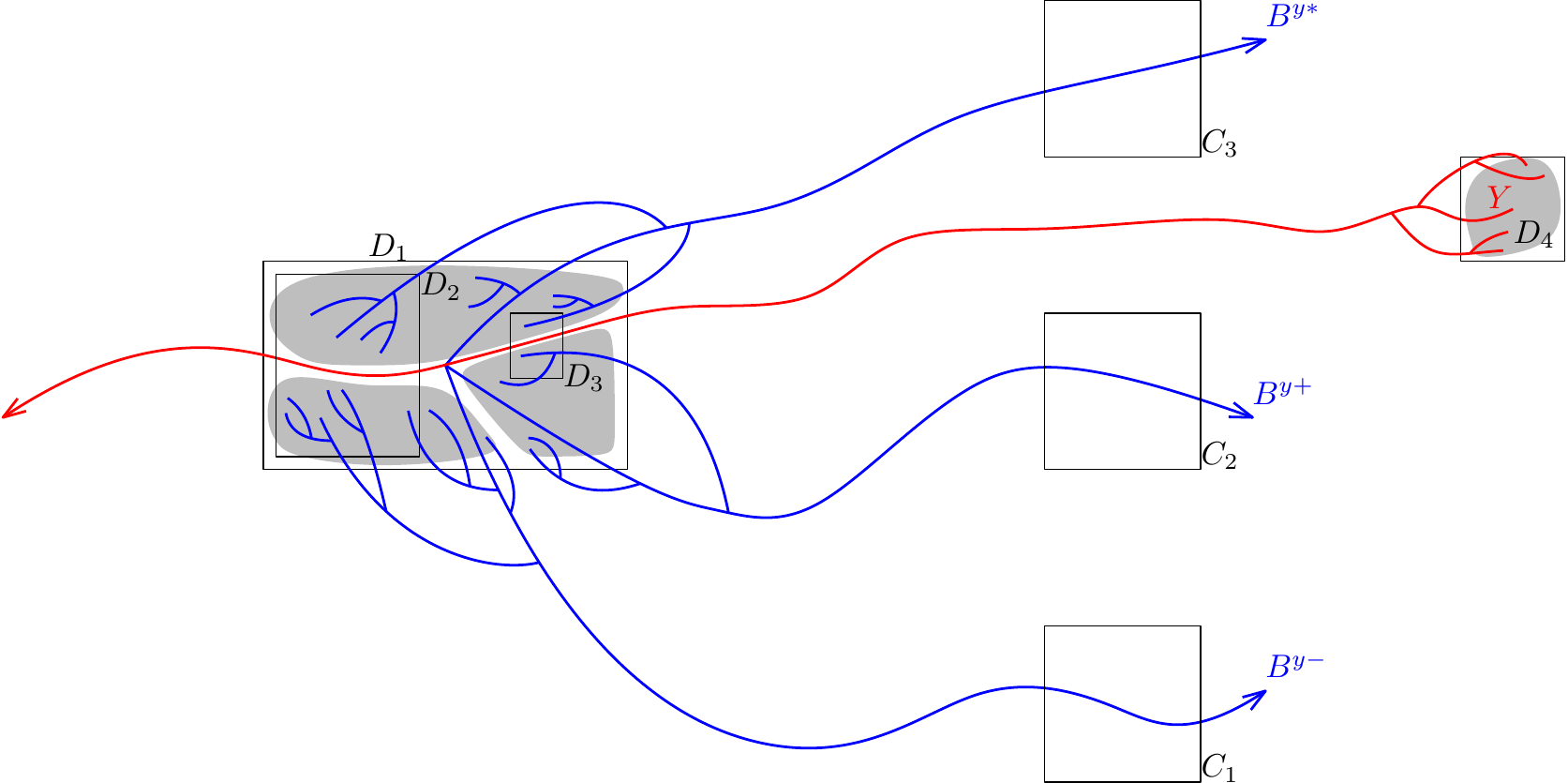}
\caption{Proof of Lemma~\ref{lemma:continuity}.}
\label{fig:continuity}
\end{figure}

We then observe that the above properties alone imply existence and uniqueness of some $y' \approx y$ and $m' \approx m$ such that $(t,y',m')\in \Psi(\W)$.
More precisely, assuming existence of boxes with these properties, there is unique $y'$ with $|y'-y|<\varepsilon$ such that $(t,y',m') \in \Psi(\W)$ for some $m'>0$. Moreover, $m'$ is unique and satisfies $|m'-m|<2\varepsilon$.
Finally, if $d(\W,\W')$ is small enough, then the same properties hold for $\W'$, which then concludes the proof.

Let us proceed to proving existence of these boxes.

We start by recalling a notion of point types. A point $(x,u)$ is of \emph{type} $(i,j)$ in a Brownian web $\W$ if there are exactly $i$ different paths \emph{arriving} at space-time point $(x,u)$ and exactly $j$ different paths \emph{departing} from $(x,u)$. Since the Brownian web is a.s.\ a complete family of coalescing paths, by reversing time the points of type $(i,j+1)$ will become $(j,i+1)$.
As shown in~\cite{FontesIsopiNewmanRavishankar04}, a Brownian web a.s.\ only contains points of types $(0,1)$, $(1,1)$, $(0,2)$, $(2,1)$, $(1,2)$, and $(0,3)$, depicted as
\\
\includegraphics[width=\textwidth]{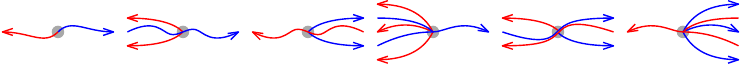}
and moreover each deterministic point is a.s.\ of type $(0,1)$.

So we can assume that $\W$ only has points of these six types.
Now observe that,  for all $t \in \t$ and $-4M < y < -2M$, we have $y\in\x(t)$ if and only if there is a dual path $A = (A_s)_{s \leq y+t} \in \W$ with $A_{y+t}<0$ coalescing with $Y$ at time $y$.
In this case, the point $(y,Y_y)$ must be of type $(0,3)$.

Suppose $(t,y,m) \in \Psi(\W)$ and let $\varepsilon>0$ be fixed.
We need to find $\delta>0$ depending on $\varepsilon$ and $\W$ such
that,
if $\W'$ is a good collection of paths with $d(\W,\W')<\delta$, then there exists unique $y'$ and $m'$ such that $|y-y'|<\varepsilon$, $|m-m'|<\varepsilon$, and $(t,y',m')\in \Psi(\W')$.
By taking first $\delta=1$, this property only depends on a finite window in the $xu$-plane, so we may assume that the metric $d$ is given by the Hausdorff distance between collections of paths, the distance between paths being itself the Hausdorff distance between sets of points.

\begin{figure}[tb]
\begin{center}
\includegraphics[page=2,width=\textwidth]{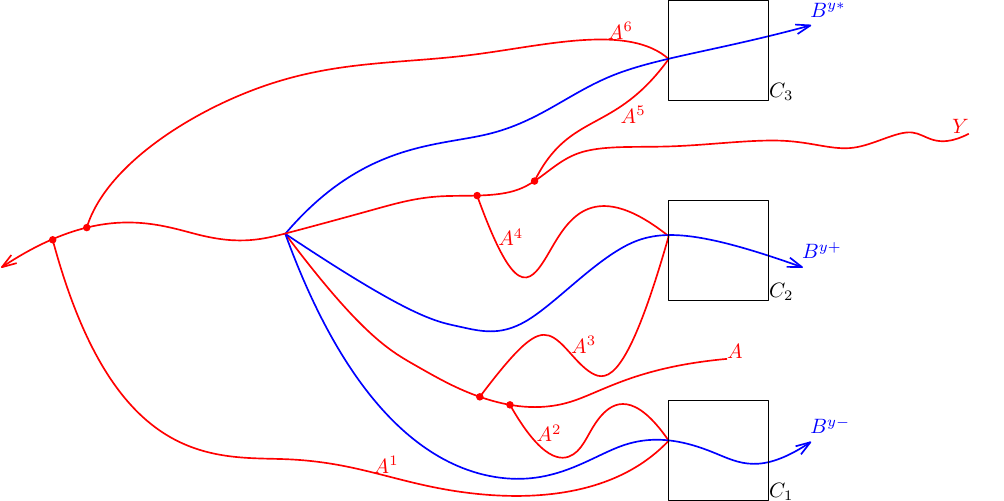}
\end{center}
\caption{Boxes $C_1,C_2,C_3$ and paths $A^1,\dots,A^6$.}
\label{fig:a1a6}
\end{figure}

We know that three paths depart form $(y,Y_y)$ in $\W$,
$B^{y-} \leqslant B^{y+}\leqslant Y \leqslant B^{y*}$, and
$B^{y+}_t-B^{y-}_t=m$.
Consider three boxes $C_1, C_2, C_3$ respectively centered at
$B^{y-}_{y+t},B^{y+}_{y+t},B^{y*}_{y+t}$ having height equal to $2\varepsilon$ and width given by
$2\varepsilon_1\leqslant2\varepsilon$, where $\varepsilon_1$ is such that each path crosses
the respective rectangle horizontally.
We say that a curve $B$ \emph{crosses the box $[a,b]\times[c,d]$ horizontally} if $c<B_x<d$ for $a \le x \le b$.
By reducing $\varepsilon$ and $\varepsilon_1$ if necessary, we can assume that such boxes do not intersect the paths $A$ and $Y$.

We now construct boxes $D_1,D_2,D_3$ shown in Fig.~\ref{fig:continuity}.
Recall that dual paths start from every point in the $xu$-plane, the set of dual paths is complete and they do not cross the primal paths.
So we can take dual paths $A^1,\dots,A^6$ as follows.
Paths $A_1$ and $A_2$ start at $(y+t-\varepsilon_1,B^{y-}_{y+t-\varepsilon_1})$, $A^1$ stays below $B^{y-}$, and $A^2$ stays above it.
Paths $A_3$ and $A_4$ start at $(y+t-\varepsilon_1,B^{y+}_{y+t-\varepsilon_1})$, $A^3$ stays below $B^{y+}$, and $A^4$ stays above it.
Paths $A_5$ and $A_6$ start at $(y+t-\varepsilon_1,B^{y*}_{y+t-\varepsilon_1})$, $A^5$ stays below $B^{y*}$, and $A^6$ stays above it.
Since the point $(y,Y_y)$ is of type $(0,3)$, and the dual path $A$ coalesces with $Y$ at $(y,Y_y)$, these paths must coalesce with either of them at some time different from $y$. Since dual paths cannot cross primal paths, $A^1$ and $A^6$ coalesce with $Y$ at some point to the left of $y$, $A^2$ and $A^3$ coalesce with $A$ at some point to the right of $y$, and $A^4$ and $A^5$ coalesce with $Y$ at some point to the right of $y$.
See Fig.~\ref{fig:a1a6}.

\begin{figure}[tb]
\centering
\includegraphics[width=.7\textwidth]{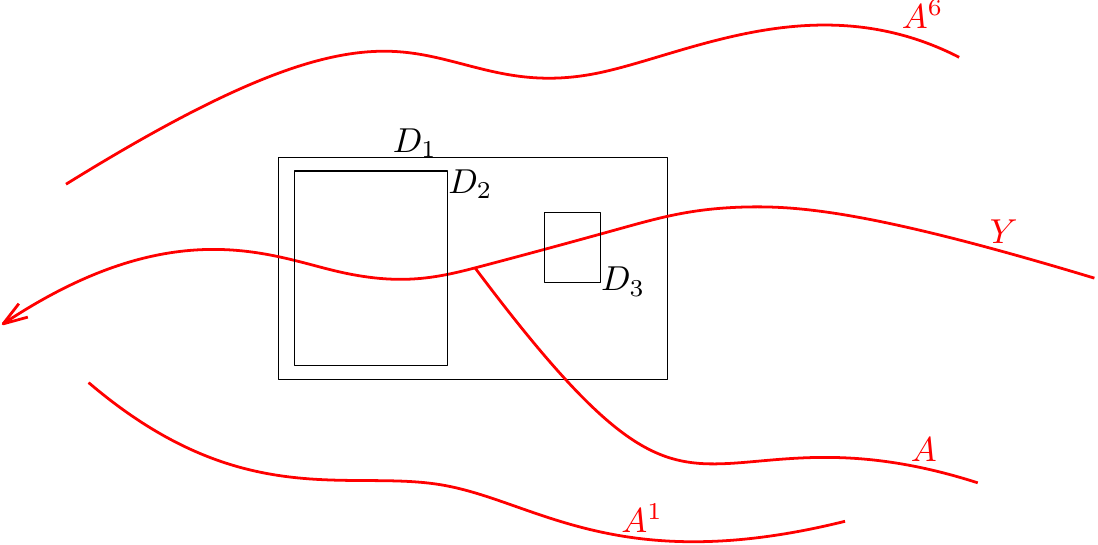}
\caption{Boxes $D_1,D_2,D_3$ between $A^1$ and $A^6$, all crossed horizontally by $Y$. Box $D_2$ is to the left of $y$, and box $D_3$ is between $A$ and $A^6$.}
\label{fig:d1d2d3}
\end{figure}

Since these curves are all continuous, we can fit a box $D_1$ centered at $(y,Y_y)$
between $A^1$ and $A^6$, and which is crossed horizontally by $Y$.
Then we pick any box $D_2$ contained in the left half of $D_1$ so that $D_2$ is crossed horizontally by $Y$.
Finally, since $Y_x > A_x$ for $x>y$, we can pick a box $D_3$ in the right half of $D_1$ so that $D_3$ is crossed horizontally by $Y$ and stays between $A$ and $A^6$.
See Fig.~\ref{fig:d1d2d3}.

The fact that primal paths cannot cross dual paths implies that all paths starting in $D_1$ will be bounced by $Y$, $A$ and $A^1,\dots,A^6$, which in turn implies that boxes $D_1,D_2,D_3$ have the desired properties shown in Fig.~\ref{fig:continuity}.

Finally, take two primal paths starting at $(-\varepsilon,Y_{-\varepsilon})$, one above $Y$ and one below $Y$. Since the origin is of type $(0,1)$, these paths are away from $0$ at $x=0$, so they coalesce at a strictly positive time and it is possible to fit a box $D_4$ centered at the origin which stays between them.
See Fig.~\ref{fig:boxd4}.
Since dual paths cannot cross primal paths, $D_4$ also has the desired properties and this concludes the proof of Lemma~\ref{lemma:continuity}.

\begin{figure}[tb]
\centering
\includegraphics[width=.5\textwidth]{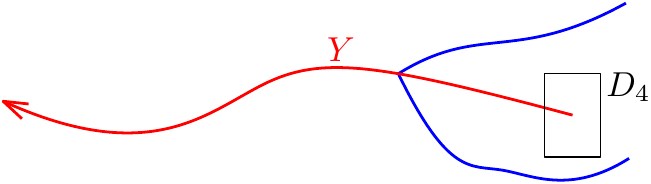}
\caption{Two primal paths starting form the same point of $Y$ near the origin, and box $D_4$ between these two paths.}
\label{fig:boxd4}
\end{figure}

%On the other hand, the point $(y,Y_y)$ is of type $(0,3)$, so every dual path that does not start on $A$ or $Y$ will coalesce with $A$ or $Y$ at some time which is different from $y$. Likewise, every primal path that starts on the path of $Y$ to the left of $(y,Y_y)$ cannot visit $(y,Y_y)$.
%In particular, $A_1$ will coalesce with $Y$ at some time $z<y$. Considering a primal path $B^1$ that starts at $(z,Y_z)$ and stays between $A_1$ and $Y$ (such path exists by completeness), we get a primal path that coalesces with $B^{y-}$ at some time $w \in (y,y+t-\varepsilon)$.

\section*{Acknowledgments}
We thank L.~R.~Fontes for motivating and inspiring discussions.

\baselineskip 11pt
\parskip 0pt
\setstretch{1}
\small
\bibliographystyle{bib/leo}

%\bibliography{bib/leo}

\begin{thebibliography}{FINR04}
\expandafter\ifx\csname urlstyle\endcsname\relax
  \providecommand{\doi}[1]{doi:\discretionary{}{}{}#1}\else
  \providecommand{\doi}{doi:\discretionary{}{}{}\begingroup
  \urlstyle{rm}\Url}\fi

\bibitem[Arr79]{Arratia79}
R.~\textsc{Arratia}.
\newblock \emph{Coalescing {Brownian} Motions on the Line}.
\newblock Ph.D. thesis, University of Wisconsin, Madison, 1979.

\bibitem[CFP07]{CaceresFerrariPechersky07}
F.~C. \textsc{C{\'a}ceres}, P.~A. \textsc{Ferrari}, E.~\textsc{Pechersky}.
\newblock \emph{A slow-to-start traffic model related to a {$M/M/1$} queue}.
\newblock J. Stat. Mech. \textbf{2007}:P07008, 2007.

\bibitem[FINR04]{FontesIsopiNewmanRavishankar04}
L.~R.~G. \textsc{Fontes}, M.~\textsc{Isopi}, C.~M. \textsc{Newman},
  K.~\textsc{Ravishankar}.
\newblock \emph{The {B}rownian web: characterization and convergence}.
\newblock Ann. Probab. \textbf{32}:2857--2883, 2004.

\bibitem[GG01]{GrayGriffeath01}
L.~\textsc{Gray}, D.~\textsc{Griffeath}.
\newblock \emph{The ergodic theory of traffic jams}.
\newblock J. Statist. Phys. \textbf{105}:413--452, 2001.

\bibitem[GKS04]{GrigorescuKangSeppaelaeinen04}
I.~\textsc{Grigorescu}, M.~\textsc{Kang}, T.~\textsc{Sepp{\"a}l{\"a}inen}.
\newblock \emph{Behavior dominated by slow particles in a disordered asymmetric
  exclusion process}.
\newblock Ann. Appl. Probab. \textbf{14}:1577--1602, 2004.

\bibitem[SK06]{SopasakisKatsoulakis06}
A.~\textsc{Sopasakis}, M.~A. \textsc{Katsoulakis}.
\newblock \emph{Stochastic modeling and simulation of traffic flow: asymmetric
  single exclusion process with {A}rrhenius look-ahead dynamics}.
\newblock SIAM J. Appl. Math. \textbf{66}:921--944, 2006.

\bibitem[SS19]{ShneerStolyar19}
S.~\textsc{Shneer}, A.~\textsc{Stolyar}.
\newblock \emph{Discrete-time {TASEP} with holdback}, 2019.
\newblock Preprint. \href{http://arxiv.org/abs/1905.03860}{arXiv:1905.03860}.

\end{thebibliography}

\end{document}